\documentclass[reqno]{amsart}

\usepackage{amsmath}
\usepackage{textcase}
\usepackage{amsthm}
\usepackage{amscd}
\usepackage{amsfonts}
\usepackage{graphics}
\usepackage[dvips]{epsfig}
\usepackage{graphicx}
\date{}

\newlength{\defbaselineskip}
\thanks{The author was partially supported by the project \emph{An\'alisis nolineal y ecuaciones diferenciales}, FQM 116, Departamento de An\'alisis Matem\'atico, Universidad de Granada}
\long\def\salta#1{\relax}

\theoremstyle{plain}
\newtheorem{theorem}{Theorem}[section]

\newtheorem{lemma}[theorem]{Lemma}

\theoremstyle{definition}
\newtheorem{remark}[theorem]{Remark}
\newtheorem{definition}[theorem]{Definition}
\theoremstyle{remark}

\newcommand{\sob}[2]{W^{#1}_{#2}(\Omega)}

\def\diff{\mathchoice{\!\setminus\!}{\!\setminus\!}{\setminus}{\setminus}}

\newcommand{\rn}{{\erre^N}}

\newcommand{\intmu}[3]{\int_{#1}{#2\,d#3}}

\newcommand{\car}[1]{\raise2pt\hbox{$\chi$}_{#1}}

\newcommand{\D}{{\nabla}}

\newcommand{\meas}[1]{{\rm meas}\,#1}

\newcommand{\pdp}{\psi_\de^+}
\newcommand{\pdm}{\psi_\de^-}

\newcommand{\fn}{f_n}
\newcommand{\fnp}{f_n^+}
\newcommand{\fnm}{f_n^-}

\newcommand{\un}{u_n}

\newcommand{\de}{\delta}

\newcommand{\la}{\lambda}

\long\def\salta#1{\relax}

\newcommand{\norma}[2]{\|#1\|_{\lower 4pt \hbox{$\scriptstyle #2$}}}

\newcommand{\ba}{\begin{array}}
\newcommand{\be}{\begin{equation}}
\newcommand{\bea}{\begin{eqnarray*}}
\newcommand{\bean}{\begin{eqnarray}}
\newcommand{\dys}{\displaystyle}
\newcommand{\ea}{\end{array}}
\newcommand{\ee}{\end{equation}}
\newcommand{\eea}{\end{eqnarray*}}
\newcommand{\eean}{\end{eqnarray}}
\newcommand{\rife}[1]{(\ref{#1})}

\catcode`\@=11

\def\@ceqnnum{{\reset@font\rm (\tempequation)}}

\def\@ceqncr{{\ifnum0=`}\fi\@ifstar{\global\@eqpen\@M
    \@cyeqncr}{\global\@eqpen\interdisplaylinepenalty \@cyeqncr}}

\def\@cyeqncr{\@ifnextchar [{\@cxeqncr}{\@cxeqncr[\z@]}}

\def\@cxeqncr[#1]{\ifnum0=`{\fi}\@@ceqncr
   \noalign{\penalty\@eqpen\vskip\jot\vskip #1\relax}}

\def\@@ceqncr{\let\@tempa\relax
    \ifcase\@eqcnt \def\@tempa{& & &}\or \def\@tempa{& &}%
      \else \def\@tempa{&}\fi
     \@tempa \if@eqnsw\@ceqnnum\fi
     \global\@eqnswtrue\global\@eqcnt\z@\cr}

\def\clabel#1#2{\global\let\tempequation #2
   \@bsphack\if@filesw {\let\thepage\relax
   \def\protect{\noexpand\noexpand\noexpand}%
   \edef\@tempa{\write\@auxout{\string
      \newlabel{#1}{{#2}{\thepage}}}}%
   \expandafter}\@tempa
   \if@nobreak \ifvmode\nobreak\fi\fi\fi\@esphack}

\def\ceqnarray{
\global\@eqnswtrue\m@th
\global\@eqcnt\z@\tabskip\@centering\let\\\@ceqncr
$$\halign to\displaywidth\bgroup\@eqnsel\hskip\@centering
  $\displaystyle\tabskip\z@{{}##}$&\global\@eqcnt\@ne
  \hskip 2\arraycolsep \hfil${{}##}$\hfil
  &\global\@eqcnt\tw@ \hskip 2\arraycolsep $\displaystyle\tabskip\z@{{}##}$\hfil
   \tabskip\@centering&\llap{##}\tabskip\z@\cr}

\def\endceqnarray{\@@ceqncr\egroup$$\global\@ignoretrue}

\@addtoreset{equation}{section}

\catcode`\@=12

\catcode`\@=11

\def\@ceqnnum{{\reset@font\rm (\tempequation)}}

\def\@ceqncr{{\ifnum0=`}\fi\@ifstar{\global\@eqpen\@M
    \@cyeqncr}{\global\@eqpen\interdisplaylinepenalty \@cyeqncr}}

\def\@cyeqncr{\@ifnextchar [{\@cxeqncr}{\@cxeqncr[\z@]}}

\def\@cxeqncr[#1]{\ifnum0=`{\fi}\@@ceqncr
   \noalign{\penalty\@eqpen\vskip\jot\vskip #1\relax}}

\def\@@ceqncr{\let\@tempa\relax
    \ifcase\@eqcnt \def\@tempa{& & &}\or \def\@tempa{& &}%
      \else \def\@tempa{&}\fi
     \@tempa \if@eqnsw\@ceqnnum\fi
     \global\@eqnswtrue\global\@eqcnt\z@\cr}

\def\clabel#1#2{\global\let\tempequation #2
   \@bsphack\if@filesw {\let\thepage\relax
   \def\protect{\noexpand\noexpand\noexpand}%
   \edef\@tempa{\write\@auxout{\string
      \newlabel{#1}{{#2}{\thepage}}}}%
   \expandafter}\@tempa
   \if@nobreak \ifvmode\nobreak\fi\fi\fi\@esphack}

\def\ceqnarray{
\global\@eqnswtrue\m@th
\global\@eqcnt\z@\tabskip\@centering\let\\\@ceqncr
$$\halign to\displaywidth\bgroup\@eqnsel\hskip\@centering
  $\displaystyle\tabskip\z@{{}##}$&\global\@eqcnt\@ne
  \hskip 2\arraycolsep \hfil${{}##}$\hfil
  &\global\@eqcnt\tw@ \hskip 2\arraycolsep
$\displaystyle\tabskip\z@{{}##}$\hfil
   \tabskip\@centering&\llap{##}\tabskip\z@\cr}

\def\endceqnarray{\@@ceqncr\egroup$$\global\@ignoretrue}

\catcode`\@=12

\long\def\salta#1{\relax}

\def\rn{\mathbb{R}^{N}}

\def\de{\delta}
\def\a#1{a(t,x,\nabla #1)}
\def\intmu#1#2#3{\int_{{#1}}{#2\,d#3}}

\def\liq{L^{\infty}(Q)}

\def\D{\nabla}

\def\be{\begin{equation}}
\def\ee{\end{equation}}

\def\rife#1{(\ref{#1})}
\def\rifer#1{({\rm \ref{#1}})}
\def\a#1{a(t,x,\nabla #1)}

\def\tkun{T_{k}(u_n)}
\def\tkunf{T_{k}(u_n-\varphi)}

\def\t1p0{\mathcal{T}^{1,p}_{0}(\Omega)}

\def\capp{\text{\rm{cap}}_{p}}

\def\m2{M^{\frac{N(p-1)}{N-1}}(\Omega)}

\def\div{{\rm div}}

\def\sob{W^{1,p}_{0}(\Omega)}

\def\lsp{\lambda_{s}^{+}}
\def\lsm{\lambda_{s}^{-}}

\newcommand{\elle}[1]{L^{#1}(\Omega)}

\newcommand{\pelle}[1]{L^{#1}(Q)}

\def\into{\int_{\Omega}}

\def\liq{L^{\infty}(Q)}
\def\w-1p'{W^{-1,p'}(\Omega)}

\def\w-1pd{W^{-1,p'}(D)}
\def\pw-1p'{L^{p'}(0,T;W^{-1,p'}(\Omega))}
\def\l{\textsl{L}}

\def\dys{\displaystyle}

\def\luq{L^{1}(Q)}
\def\lp'n{(L^{p'}(\Omega))^{N}}

\def\fnp{f^{\oplus}_{n}}
\def\fnm{f^{\ominus}_{n}}
\def\supp{\text{\rm{supp}}}

\def\pdp{\psi_{\delta}^{+}}

\def\pdm{\psi_{\delta}^{-}}

\def\pep{\psi_{\eta}^{+}}
\def\pem{\psi_{\eta}^{-}}

\def\kdp{K^{+}_{\delta}}
\def\kdm{K^{-}_{\delta}}

\def\pep{\psi_{\eta}^{+}}
\def\pem{\psi_{\eta}^{-}}

\def\psob{L^{p}(0,T;W^{1,p}_{0}(\Omega))}

\def\l2h-1{L^2 (0,T ; H^{-1} ( \Omega ))}

\def\hmun{h_m (u_n)}

\def\para#1#2{L^{#1}(0,T;W^{1,#2}_0(\Omega))}

\def\meas#1{{\rm meas}\,#1}

\def\tom{{Q}}

\def\paraduale{L^{p'}(0,T;\duale)}

\def\duale{W^{-1,p'}(\Omega)}

\def\car#1{\raise2pt\hbox{$\chi$}_{#1}}

\newcommand{\LL}{\>\hbox{\vrule width.2pt \vbox to 7pt{\vfill\hrule width 7pt height.2pt}}\>}
\def\supp{\text{\rm{supp}}}

\def\div{\begin{rm}div\end{rm}}

\def\liq{L^{\infty}(Q)}

\def\lp'n{(L^{p'}(\Omega))^{N}}

\def\sob{W^{1,p}_{0}(\Omega)}
\def\t1p0{T^{1,p}_{0}(\Omega)}
\def\w-1p'{W^{-1,p'}(\Omega)}
\def\pw-1p'{L^{p'}(0,T;W^{-1,p'}(\Omega))}
\def\psob{L^{p}(0,T;W^{1,p}_{0}(\Omega))}
\def\lil2{L^{\infty}(0,T;L^2 (\Omega))}

\def\l2h10{L^2 (0,T ; H^1_0 ( \Omega ))}

\def\m2{M^{\frac{N(p-1)}{N-1}}(\Omega)}

\def\pdp{\psi_{\delta}^{+}}
\def\pdm{\psi_{\delta}^{-}}
\def\pep{\psi_{\eta}^{+}}
\def\pem{\psi_{\eta}^{-}}

\def\tkun{T_k (u_n)}

\def\capp{\text{\rm{cap}}_{p}}

\def\l{\textsl{L}}

\def\de{\delta}
\def\D{\nabla}

\def\intmu#1#2#3{\int_{{#1}}{#2\,d#3}}

\def\into{\int_{\Omega}}

\def\intq{\int_{Q}}

\def\intq1{\displaystyle \int_{\Omega \times (0, 1)}}

\def\rife#1{(\ref{#1})}
\def\rifer#1{({\rm \ref{#1}})}

\def\dys{\displaystyle}

\author[F. Petitta]{Francesco Petitta}

\address[F. Petitta]{Dipartimento di Scienze di Base e Applicate
per l' Ingegneria, ``Sapienza", Universit\`a di Roma, Via Scarpa 16, 00161 Roma, Italy.}\email{francesco.petitta@sbai.uniroma1.it}

\begin{document}

%\setlinespacing{1}

\title[Nonexistence for parabolic problems with singular measures]{A nonexistence result  for nonlinear parabolic equations with singular measures as data}

\begin{abstract}
In this paper we prove a nonexistence result for nonlinear parabolic problems with zero lower order term whose model is
$$
\begin{cases}
    u_{t}- \Delta_p u+|u|^{q-1}u=\lambda & \text{in}\ (0,T)\times\Omega \\
    u(0,x)=0 & \text{in}\ \Omega,\\
u(t,x)=0 & \text{on}\  (0,T)\times\partial\Omega,
\end{cases}
$$
where $\Delta_p =\div(|\nabla u|^{p-2}\nabla u)$ is the usual $p$-laplace operator,
 $\lambda$ is measure concentrated on a set of zero parabolic $r$-capacity ($1<p<r$),  and $q$ is large enough.
\end{abstract}

\maketitle
\section{Introduction}

The question whether a solution should exists or not for semilinear problems
has been largely studied in the elliptic framework; in a pioneering paper by H.
Brezis (\cite{B}) the author proved the following
\begin{theorem}\label{dirac}
Let $\Omega$ be a bounded open subset of $\rn$, $N > 2$, with $0 \in \Omega$, let $f$ be a
function in $\elle1$, and let $\fn$ be a sequence of $\elle\infty$ functions such that
\begin{equation}
\lim_{n \to +\infty} \, \intmu{\Omega \diff B_\rho(0)}{|\fn-f|}{x} = 0\,,
\qquad \forall \rho > 0\,.
\label{condbre}
\end{equation}
Let $\un$ be the sequence of
solutions of the following nonlinear elliptic problems
\begin{equation}\label{pb}
\begin{cases}
-\Delta \un + |\un|^{q-1}\,\un = \fn&\ \text{in }\ \Omega\\
\un = 0, &\text{on }\ \partial\Omega
\end{cases}
\end{equation}
with $q \geq \frac{N}{N-2}$. Then $\un$ converges to the unique solution $u$ of the equation
$-\Delta u + |u|^{q-1}\,u = f$.
\end{theorem}

If $f = 0$, an example of functions $\fn$ satisfying condition \rife{condbre} is that of a sequence
of nonnegative $\elle\infty$ functions converging in the weak$^\ast$ topology of measures to
$\delta_0$, the Dirac mass concentrated at the origin. In this case, $\un$ converges to zero, which is not a solution of the equation with $\delta_0$ as datum. The
result of Theorem \ref{dirac} is strongly connected with a theorem by P. B\'enilan and H. Brezis (see
\cite{bb}), which states that the problem $-\Delta u + |u|^{q-1}\,u = \delta_0$ has no distributional
solution if $q \geq \frac{N}{N-2}$. On the other hand (see \cite{BBC} and \cite{B}), if $q <
\frac{N}{N-2}$, then there exists a unique solution of
$$
\begin{cases}
-\Delta u + |u|^{q-1}\,u = \de_0,&\text{in}\ \ \Omega\\
u = 0&\ \text{on}\ \ \partial\Omega.
\end{cases}
$$

The \emph{threshold} $\frac{N}{N-2}$ essentially depends on  the linearity of the
laplacian operator, and  on the
fact that the Dirac mass is a measure which is concentrated on a point: a set of zero elliptic $N$-capacity.

In \cite{op} this result was improved to the nonlinear framework; there  the authors  actually proved that, if $\lambda$ is a measure concentrated on a set of zero elliptic $r$-capacity, $r<q$, and $q$ is large enough, then
 problem
$$
\begin{cases}
-\Delta_p u + |u|^{q-1}\,u = \lambda,&\text{in}\ \ \Omega\\
u = 0&\ \text{on}\ \ \partial\Omega,
\end{cases}
$$
has no solutions in a very \emph{strong sense}; that is, if we approximate $\lambda$ with smooth functions in the narrow topology of measures then the approximating solutions $u_n$ converge to $0$.
In the same paper the result is proved for more general Leray-Lions type nonlinear operators (see \cite{ll}).

\medskip

In this paper, we will combine an idea of \cite{op} with a suitable parabolic \emph{cut-off lemma} to prove a general nonexistence  result in the framework of nonlinear parabolic problems with singular measures as data.
\medskip

If $\Omega$ is an open bounded subset of $\rn$, $N>2$, and $T>0$  we denote by
$Q$ the parabolic cylinder $(0,T)\times\Omega$. If $\lambda$ a bounded Radon
measure on $Q$, then we will say that $\lambda$ is concentrated on a Borel set
$B$, and write $\la=\la_{\LL B}$, if $\la(E)=\la(B\cap E)$, for any measurable subset $E$ of $Q$.

Our main result (see Theorem \ref{none} below) states nonexistence of solutions for parabolic problems in the sense of approximating sequences; as a particular case of it we will obtain the following:

\begin{theorem}\label{noneint}
Let  $f_n$  be a sequence of functions in $\liq$ such that 
 $$
\lim_{n\to\infty}\int_Q\varphi\ f_n\ dx = \int_Q\varphi\ d\lambda, \ \ \ \forall \ \varphi\in C(\overline{Q}), 
$$
where $\lambda$ is a bounded Radon measure on $Q$ concentrated on a
set of zero parabolic $r$-capacity, and let
\begin{equation}\label{assint}
q>\frac{r}{r-2}\,.
\end{equation}
Then the solutions of
\begin{equation}\label{lim1int}
\begin{cases}
    (u_{n})_{t}- \Delta u_n+|u_n|^{q-1}u_n=f_n & \text{in}\ (0,T)\times\Omega \\
    u_n(0,x)=0 & \text{in}\ \Omega,\\
u_n(t,x)=0 & \text{on}\  (0,T)\times\partial\Omega,
\end{cases}
\end{equation}
are such that both $u_n$ and $|\nabla u_n|$  converge to $0$ in $\pelle{1}$.

Moreover,
$$
\lim_{n\to\infty}\int_{Q} |u_n |^{q-1}u_n\varphi\ dx=\int_{Q} \varphi \ d\lambda\,,\ \ \forall\ \varphi\in C_{0}(Q).
$$
\end{theorem}
\begin{remark}
Theorem \ref{noneint} states that in fact the sets of zero $r$-capacity are in some sense \emph{removable singularities} for problem
\begin{equation}\label{ne1int}
\begin{cases}
    u_{t}- \Delta u+|u|^{q-1}u=f & \text{in}\ (0,T)\times\Omega \\
    u(0,x)=0 & \text{in}\ \Omega,\\
u(t,x)=0 & \text{on}\  (0,T)\times\partial\Omega,
\end{cases}
\end{equation}
 with large $q$, since the approximation does not see them.  In fact,  the singular measure $\la$ turns out to be \emph{cancelled out}  by the zero order terms of the approximating problems in the  weakly$^\ast$ sense of the measures.

Moreover, as we shall prove, the convergence is actually stronger than the one stated in Theorem \ref{noneint}.

 Let us finally explicitly remark that the choice of the homogeneous initial datum is not restrictive; indeed, since the result is obtained for measures on $Q$ which do not charge the set $\{0\}\times\Omega$ then our argument is, as we will see, essentially independent on the initial datum.
\end{remark}

\section{Basic assumptions and tools}

 Let $p>1$; we  recall the notion of \emph{parabolic $p$-capacity} associated
to our problem (for further details see
\cite{pierre}, \cite{dpp}).

\begin{definition}\label{cappara1}
Let $Q=Q_{T}=(0,T)\times\Omega$ for any fixed $T>0$, and let us
define $V=\sob \cap L^{2}(\Omega)$, endowed with its natural norm
$\|\cdot\|_{\sob}+\|\cdot\|_{L^{2}(\Omega)}$ and
\begin{equation}\label{W}
W=\left\{ u\in L^{p}(0,T;V),\ u_{t}\in L^{p'}(0,T;V') \right\},
\end{equation}
endowed with its natural norm $\|u\|_{W}=\|u\|_{ L^{p}(0,T;V)}+
\|u_{t}\|_{ L^{p'}(0,T;V')}$. If $U\subseteq Q$ is an open set, we
define the \emph{parabolic $p$-capacity} of $U$ as
\[
\capp(U)=\inf\{\|u\|_{W}:u\in W,u\geq \chi_{U}\  \text{a.e. in}\  Q\},
\]
where as usual we set $\inf\emptyset=+\infty$; we then define for any Borel
set $B\subseteq Q$
\[
\capp(B)=\inf\{\capp(U), U \ \text{open set of}\ Q, B\subseteq U\}.
\]
\end{definition}

Let us state our basic assumptions; let $\Omega$ be a bounded, open subset of $\rn$, $T$ a positive number and
$Q=(0,T)\times\Omega$. Let $a :(0,T) \times\Omega
\times \rn \to \rn$ be a Carath\'eodory function (i.e., $a(\cdot,\cdot,\xi)$
is measurable on $\tom$ for every $\xi$ in $\rn$, and $a(t,x,\cdot)$ is
continuous on $\rn$ for almost every $(t,x)$ in $\tom$), such that the
following holds:
\be
a(t,x,\xi)  \cdot\xi \geq \alpha\,|\xi|^p\,,\quad p>1\,,
\label{coercp}
\ee
\be
|a(t,x,\xi)| \leq \beta\,[b(t,x) + |\xi|^{p-1}]\,,
\label{cont}
\ee
\be
[a(t,x,\xi) - a(t,x,\eta)] (\xi - \eta) > 0\,,
\label{monot}
\ee
for almost every $(t,x)$ in $Q$, for every $\xi$, $\eta$ in $\rn$, with
$\xi \neq \eta$, where
$\alpha$ and $\beta$ are two positive constants, and
$b$ is a nonnegative function in $L^{p'}(\tom)$.

We define the differential operator
$$
A(u) = -\div (a(t,x,\D u))\,, \qquad u \in \psob\,.
$$
Under assumptions \rife{coercp}, \rife{cont} and
\rife{monot}, $A$ is a
 coercive and pseudo\-monotone operator acting from
the space $L^{p}(0,T;\sob)$ into its dual
$\paraduale$.

We deal with problem

\begin{equation}\label{ne1}
\begin{cases}
    u_{t}- \div(a(t,x,\nabla u))+|u|^{q-1}u=g +\lambda & \text{in}\ (0,T)\times\Omega \\
    u(0,x)=0 & \text{in}\ \Omega,\\
u(t,x)=0 & \text{on}\  (0,T)\times\partial\Omega,
\end{cases}
\end{equation}
with $g\in\pelle1$, $q>1$, $a$ satisfying \rife{coercp}, \rife{cont} and \rife{monot}, and $\lambda=\lambda^+ -\lambda^-$ is a bounded measure concentrated on a set $E=E^+ \cup E^-$, such that cap$_{r}$$(E)=0$.

Let us mention that existence of  renormalized solutions (which in
particular turn out to be distributional solutions for problem \rife{ne1}) is
one of the results proved in a forthcoming paper (see \cite{ppp}) in the case
of diffuse measures as data, that is measures which does not charge the sets of
zero parabolic $p$-capacity.

Let us recall that a sequence of bounded measures $\lambda_n$ on an open set  $D\subset\rn$ narrowly converges to a measure $\lambda$ if
 $$
\lim_{n\to\infty}\int_D\varphi\ d\lambda_n = \int_D\varphi\ d\lambda, \ \ \ \forall \ \varphi\in C(\overline{D}).
$$

 We approximate the data with smooth $g_n$ which converge to $g$ in $\pelle1$ and smooth $f_n=f_{n}^{\oplus} -f_{n}^{\ominus}$, with  $f_{n}^{\oplus}$ and $f_{n}^{\ominus}$  converging, respectively, to $\lambda^+$ and  $\lambda^-$ in the narrow topology of measures. We consider the solutions $u_n$ of
\begin{equation}\label{ane1}
\begin{cases}
    (u_{n})_t- \div(a(t,x,\nabla u_n ))+|u_n|^{q-1}u_n=g_n +f_n & \text{in}\ (0,T)\times\Omega \\
    u_n (0,x)=0 & \text{in}\ \Omega,\\
u_n (t,x)=0 & \text{on}\  (0,T)\times\partial\Omega.
\end{cases}
\end{equation}

 Let us give the notion of entropy solution for parabolic problem \rife{ne1} with a general $g\in\elle1$, recalling that
 \[
S^p = \{ u\in \psob; u_t \in\pw-1p' +  \luq \},
\]
that  $T_{k}(s)=\text{max}(-k,\text{min}(k,s))$ for any  $k>0$, and that 
\[
\Theta_{k}(z)=\int_{0}^{z}T_{k}(s)\ ds,
\]
is  the primitive of the truncation function.
 \begin{definition}\label{edef}
Let $g\in\elle1$ and $\lambda=0$. A
measurable function $u$ is an \emph{entropy solution} of \rife{ne1} if
\be\label{regug2e}
 T_k(u-g)\in \para pp\mbox{ for every $k>0$, }
\ee
\be\label{rasye} t\in [0,T]\mapsto\into\Theta_k (u-g-\varphi)(t,x)\ dx
\ee
is a continuous function for all $k\geq 0$ and all $\varphi \in S^p \cap \liq$, and moreover
\be
\label{reqee}
\begin{array}{l}
 \dys\into \Theta_k (u-g-\varphi)(T,x)\ dx -\into \Theta_k (u-g-\varphi)(0,x)\ dx\\\\
 \dys\quad
 \dys+\int_{0}^{T}\langle\varphi_t , T_k (u-g -\varphi)\rangle\ dt\\\\
 \dys+\int_Q \a{u}\cdot\nabla T_k (u-g-\varphi)\ dxdt \\\\
 \dys\leq \int_Q g T_k (u-g-\varphi)\ dxdt,
\end{array}
\ee
for all $k\geq 0$ and all $\varphi \in S^p \cap \liq$.
\end{definition}

Recall that, thanks to a result of \cite{dp}, the unique entropy solution of problem  \rife{ne1} (with $\la=0$) turns out to coincide with the renormalized solution of the same problem as introduced in \cite{pr} (see also \cite{dpp} and \cite{pe2}).

\medskip

As we said, our main result concerns the nonexistence of solutions for problem \rife{ne1} in the sense of approximating sequences; let us state it.
\begin{theorem}\label{none}
Let $1<p<r$, and
\begin{equation}\label{ass}
q>\frac{r(p-1)}{r-p}, 
\end{equation}
and let $u_n$ be the unique solution of problem \rife{ane1}. 
Then $|\nabla u_n|^{p-1}$ converges strongly to $|\nabla u |^{p-1}$ in $\pelle{\sigma}$ with $\dys \sigma<\frac{pq}{(q+1)(p-1)}\ $, where $u$ is the unique entropy (renormalized) solution of problem
\begin{equation}\label{lim1}
\begin{cases}
    u_{t}- \div(a(t,x,\nabla u))+|u|^{q-1}u=g & \text{in}\ (0,T)\times\Omega \\
    u(0,x)=0 & \text{in}\ \Omega,\\
u(t,x)=0 & \text{on}\  (0,T)\times\partial\Omega.
\end{cases}
\end{equation}
Moreover,
\begin{equation}\label{rico}
\lim_{n\to\infty}\int_{Q} |u_n |^{q-1}u_n\varphi\ dx=\int_{Q} |u |^{q-1}u\varphi\ dx+\int_{Q} \varphi \ d\lambda\,,\ \ \forall\ \varphi\in C_{0}(Q).
\end{equation}
\end{theorem}

\section{Proof of Theorem \ref{none}}

In this section we prove Theorem \ref{none}. From here on $\omega$ will  indicate any quantity that vanishes as the parameters in its argument go to
their (obvious, if not explicitly stressed) limit point with the same order in which they appear,  that is, as an example
\[
\dys\lim_{\delta\rightarrow 0^+}\limsup_{m\rightarrow +\infty}
\limsup_{n\rightarrow \infty} |\omega(n,m,\delta)|=0.
\]
Moreover, for the sake of
simplicity, in what follows, the convergences, even if not explicitly stressed, may be understood
to be taken possibly up to a suitable subsequence extraction.

To prove Theorem \ref{none} we will use the following Lemma proved in \cite{pe2}.
\begin{lemma}\label{acp}
Let $\mu=\lsp-\lsm$ be a bounded Radon measure on $Q$, where $\lsp$ and $\lsm$ are  nonnegative and concentrated, respectively, on two disjoint sets $E^+$ and $E^-$  of zero $r$-capacity. Then, for every $\delta>0$, there exist two compact sets $\kdp\subseteq E^+$ and $\kdm\subseteq E^-$ such that

\begin{equation} \label{acp1}
\lsp(E^+\backslash \kdp)\leq \delta, \ \ \ \ \lsm(E^-\backslash \kdm)\leq \delta,
\end{equation}
and there exist $\pdp, \ \pdm\in  C^{1}_{0}(Q)$, such that
\begin{equation}\label{acp2}
\pdp,\ \pdm\equiv 1\ \text{respectively on}\ \ \kdp, \ \kdm,
\end{equation}
\begin{equation}\label{acp3}
0\leq\pdp,\ \pdm\leq 1,
\end{equation}
\begin{equation}\label{acp4}
\supp(\pdp) \cap \supp(\pdm)\equiv \emptyset.
\end{equation}
Moreover
\begin{equation}\label{acp5}
\|\pdp\|_{S^r}\leq \delta,\ \ \ \|\pdm\|_{S^r}\leq \delta,
\end{equation}
and, in particular, there exists a decomposition of $(\pdp)_{t}$ and a decomposition  of $(\pdm)_{t}$ such that
\begin{equation}\label{acp6}
\dys\|(\pdp)_{t}^{1}\|_{L^{r'}(0,T;W^{-1,r'}(\Omega))}\leq \delta,\ \ \ \|(\pdp)_{t}^{2}\|_{\luq}\leq \delta,
\end{equation}
\begin{equation}\label{acp7}
\dys\|(\pdm)_{t}^{1}\|_{L^{r'}(0,T;W^{-1,r'}(\Omega))}\leq \delta,\ \ \ \|(\pdm)_{t}^{2}\|_{\luq}\leq \delta,
\end{equation}
and both $\pdp$ and $\pdm$ converge to zero weakly$^\ast$ in $\liq$, in $\luq$, and, up to subsequences, almost everywhere as $\delta$ vanishes.

Moreover, if $f_n= \fnp-\fnm$ is as in \rife{ane1}, we have
\begin{equation}\label{acp8}
\int_{Q}\pdm \fnp=\omega(n,\delta),\ \ \ \ \int_{Q}\pdm\ d\lsp\leq\delta ,
\end{equation}
\begin{equation}\label{acp9}
\int_{Q}\pdp \fnm=\omega(n,\delta),\ \ \ \ \int_{Q}\pdp\ d\lsm\leq\delta ,
\end{equation}
\begin{equation}\label{acp10}
\int_{Q}(1-\pdp)\fnp=\omega(n,\delta),\ \ \ \ \int_{Q}(1-\pdp)\ d\lsp\leq \delta,
\end{equation}
\begin{equation}\label{acp11}
\int_{Q}(1-\pdm) \fnm=\omega(n,\delta),\ \ \ \ \int_{Q}(1-\pdm)\ d\lsm\leq \delta.
\end{equation}
\end{lemma}

\medskip

 For the convenience of the reader we will split the proof of Theorem \ref{none} in three steps. In the first one we prove some basic estimates on the approximating solutions, while the second step is devoted to check how the zero order term behaves far from the support of $\lambda$; finally, in the third step we conclude the proof by showing that the limit function $u$ is an entropy solution of  problem \rife{lim1} and \rife{rico} holds true.

\begin{proof}[Proof of Theorem \ref{none}]
\noindent {\bf Step $1$.} Basic estimates.

Taking $T_k (u_n)$ as test function in the weak formulation of \rife{ane1}, we readily have the following estimates on the approximating solutions:
\begin{equation}\label{tkun}
\int_{Q}|\nabla \tkun|^p \leq Ck,
\end{equation}
\begin{equation}\label{cl1}
\sup_{t}\into |u_n |\leq C
\end{equation}
and moreover, since,
\[
k\int_{\{|u_n |\geq k\}}|u_n |^q\leq \int_{Q}|u_n|^q |\tkun|\leq Ck,
\]
so that
\[
k^q\meas{\{|u_n |\geq k\}}\leq C,
\]
and
\[
\int_{\{|u_n |<k \}}|u_n |^q\leq C k^q,
\]
we have
\[
|u_n|^q \ \ \text{is bounded in}\ \  \pelle{1}.
 \]
 Because of this fact, using \rife{tkun}, one can prove, reasoning as in \cite{bdgo},
 \[
 |\nabla u_n |^{p-1} \ \text{is bounded in } L^\rho (Q),  \ \dys \ \text{for any }\ \rho<\frac{pq}{(q+1)(p-1)}.
 \]
 Moreover $u_n$ (up to subsequences) converges almost everywhere to a function $u$, and, looking at the equation in \rife{ane1}, we have that
\[
 (u_{n})_t- \div(a(t,x,\nabla u_n ))
\]
is bounded in $\pelle{1}$ and so by Theorem $3.3$ of \cite{bdgo} we have that
\[
\nabla u_n \longrightarrow \nabla u\ \ \text{a.e. on}\  Q.
\]
Therefore, thanks to the growth condition on $a$, we have that both
\begin{equation}
|\nabla u_n|^{p-1}\longrightarrow |\nabla u|^{p-1}\ \ \text{strongly in $(L^{\rho}(Q))^N$}
\end{equation}
and
\begin{equation}\label{conv}
a(t,x,\nabla u_n)\longrightarrow a(t,x,\nabla u)\ \ \text{strongly in $(L^{\rho}(Q))^N$}
\end{equation}
for every $\rho<\frac{pq}{(q+1)(p-1)}\ $.

\noindent {\bf Step $2$.} Energy estimates. 

Let $\Psi_\delta =\pdp+\pdm$, as in Lemma \ref{acp}; let us mention that the use of these type of cut-off functions to deal with, separately, 
the regular and the singular  part of the data was first introduced in \cite{dmop} in the elliptic framework.

Then, we want to show that
\begin{equation}\label{q1}
\int_{\{u_n>2m\}}|u_{n}|^{q}(1-\Psi_\delta )\ dx =\omega(n,m,\delta),
\end{equation}
and
\begin{equation}\label{q2}
\int_{\{u_n<-2m\}}|u_{n}|^{q}(1-\Psi_\delta )\ dx =\omega(n,m,\delta).
\end{equation}
We will prove \rife{q1} (the proof of \rife{q2} is analogous).
Let us define
\begin{equation}\label{bm}
\dys \beta_m (s)=
\begin{cases}
1 & \text{if}\ \ s> 2m,\\
\dys\frac{s}{m}-1&\text{if}\ \ m<s\leq 2m,\\
0&\text{if}\ \ s\leq m.
\end{cases}
\end{equation}
and let us take $\beta_m (u_n)(1-\Psi_\delta)$ as test function in \rife{ane1}; we obtain
\begin{ceqnarray}
&&
\dys\int_{0}^{T}\langle (u_{n})_t,\beta_m (u_n)(1-\Psi_\delta)\rangle\ dt \clabel{at}{A}\\
&&
\qquad+ \dys \frac{1}{m}\int_{\{m<u_n\leq 2m\}}\a{u_n}\cdot\nabla u_n (1-\Psi_\delta)\clabel{a}{B}\\
&&
\quad -\dys\int_{Q}\a{u_n }\cdot\nabla \Psi_\delta \beta_m (u_n)\clabel{b}{C}\\
&&
\qquad\dys+\int_{Q}|u_n|^{q-1}u_n  \beta_m (u_n)(1-\Psi_\delta) \clabel{c}{D}\\
&&
 =\dys \int_{Q}\fnp\beta_m (u_n)(1-\Psi_\delta)\clabel{d}{E}\\
&&
\quad\dys-\int_{Q}\fnm \beta_m (u_n)(1-\Psi_\delta)\clabel{e}{F}\\
 &&
 \dys\qquad+\int_{Q}g_n  \beta_m (u_n)(1-\Psi_\delta).\clabel{f}{G}
\end{ceqnarray}
Let us analyze all terms one by one. Using \rife{conv} and  assumption \rife{ass}, by means of Egorov Theorem we readily have
\[
-\rifer{b}=\omega(n,m),
\]
and, again by Egorov Theorem we get
\[
\rifer{f}=\omega(n,m).
\]
On the other hand, thanks to  Lemma \ref{acp}, we can write
\[
\begin{array}{l}
\dys \rifer{d}\leq\int_{Q}\fnp (1-\Psi_\delta)\ dx=\int_{Q}\fnp (1-\pdp)\ dx +\int_{Q}\fnp \pdm\ dx\\
=\dys\int_{Q} (1-\pdp)\ d\lambda^+ +\int_{Q} \pdm\ d\lambda^- +\omega(n)=\omega(n,\delta).
\end{array}
\]
Moreover, we can drop both \rifer{a} and $-$\rifer{e} since they are nonnegative, while, if $B_m$ is the primitive function of $\beta_m$, we can write
\[
\begin{array}{l}
\dys \rifer{at}=\int_{Q}B_m (u_n)_t (1-\Psi_\delta)\\\\
\dys=\int_{Q}B_m (u_n)(\Psi_{\delta})_t+\into B_m (u_n)(T)\geq \omega(n,m).
\end{array}
\]
Collecting together all these results we obtain \rife{q1}.

\noindent {\bf Step $3$.} Passing to the limit.

Here, for technical reasons, we use of the double cut-off function $\Psi_{\delta,\eta}=\pdp\pep+\pdm\pem$ where $\pdp,\pdm,\pep,\pem$ are the functions  constructed in Lemma \ref{acp}; the same trick has been also used in \cite{pe2} (see also \cite{dmop}).

 Let us define
\begin{equation}\label{hm}
\dys h_m (s)=
\begin{cases}
0 & \text{if}\ \ |s|> 2m,\\
\dys 2-\frac{|s|}{m}, &\text{if}\ \ m<|s|\leq 2m,\\
1&\text{if}\ \ |s|\leq m.
\end{cases}
\end{equation}

 We take $\tkunf(1-\Psi_{\delta,\eta})\hmun$ in the weak formulation of \rife{ane1}, and  we have
\begin{ceqnarray}
&&
\dys\int_{0}^{T}\langle (u_{n})_t,\tkunf(1-\Psi_{\delta,\eta})\hmun\rangle\ dt \clabel{a1}{A}_t\\
&&
\qquad+ \dys \int_{Q}\a{u_n }\cdot\nabla \tkunf(1-\Psi_{\delta,\eta})\hmun\clabel{b1}{B}\\
&&
\quad -\dys\int_{Q}\a{u_n }\cdot\nabla \Psi_{\delta,\eta}\tkunf\hmun\clabel{c1}{C}\\
&&
\qquad\dys+\int_{Q}|u_n|^{q-1}u_n \tkunf(1-\Psi_{\delta,\eta})\hmun \clabel{d1}{D}\\
&&
 =\dys \int_{Q}\fnp\tkunf(1-\Psi_{\delta,\eta})\hmun\clabel{e1}{E}\\
&&
\quad\dys-\int_{Q}\fnm \tkunf(1-\Psi_{\delta,\eta})\hmun\clabel{f1}{F}\\
 &&
 \dys\qquad+\int_{Q}g_n  \tkunf(1-\Psi_{\delta,\eta})\hmun\clabel{g1}{G}\\
 &&
 \qquad- \dys \frac{1}{m}\int_{\{m<u_n\leq 2m\}}\a{u_n}\cdot\nabla u_n (1-\Psi_{\delta,\eta})\tkunf\clabel{h1}{H}\\
 &&
\qquad+ \dys \frac{1}{m}\int_{\{-2m\leq u_n< -m\}}\a{u_n}\cdot\nabla u_n (1-\Psi_{\delta,\eta})\,\tkunf \,. \clabel{i1}{I}
\end{ceqnarray}

Using Lemma \ref{acp} and \rife{conv} we have \rifer{c1}$=\omega(n,\eta)$, while
\[
|\rifer{e1}|+|\rifer{f1}|\leq k\int_{Q}(\fnp +\fnm)(1-\Psi_{\delta,\eta})\ dx =\omega(n,\eta),
\]
and easily
\[
\rifer{g1}=\int_{Q}g T_k (u-\varphi)\ dx +\omega(n,\eta).
\]

On the other hand, using Lemma 6 of \cite{pe2} we deduce that  $|\rifer{h1}| + |\rifer{i1}|= \omega(n,m,\eta)$.

Now let us look at \rifer{d1}:
\[
\begin{array}{l}
\rifer{d1}=\dys\int_{\{-2m\leq  u_n \leq 2m\}}|u_n|^{q-1}u_n\tkunf(1-\Psi_{\delta,\eta})\hmun\\\\
\dys+ \int_{\{u_n > 2m\}}u_n^{q}\tkunf(1-\Psi_{\delta,\eta})\hmun \\\\
\dys+\int_{\{  u_n <- 2m\}}|u_n|^{q}\tkunf(1-\Psi_{\delta,\eta})\hmun.
\end{array}
\]
Using \rife{q1} and \rife{q2} we have that the last two terms in the right hand side are $\omega(n,m,\eta)$, while
\[
\begin{array}{l}
\dys\int_{\{-2m\leq  u_n \leq 2m\}}|u_n|^{q-1}u_n\ \tkunf(1-\Psi_{\delta,\eta})\hmun\\\\
=\dys \int_{\{-2m\leq  u \leq 2m\}}|u|^{q-1}u \ T_k (u-\varphi)(1-\Psi_{\delta,\eta})\hmun+\omega(n)\\\\
\dys= \int_{Q}|u|^{q-1}u \ T_k (u-\varphi)(1-\Psi_{\delta,\eta})+\omega(n,m)\\\\
\dys=\int_{Q}|u|^{q-1}u \ T_k (u-\varphi)+\omega(n,m,\eta).
\end{array}
\]
So that
\[
\rifer{d1}=\int_{Q}|u|^{q-1}u \ T_k (u-\varphi)+\omega(n,m,\eta).
\]
Moreover,
\[
\begin{array}{l}
\rifer{b1}= \dys \int_{Q}[\a{u_n }-\a{\varphi}]\cdot\nabla \tkunf(1-\Psi_{\delta,\eta})\hmun\\\\
+\dys \int_{Q}\a{\varphi}\cdot\nabla\tkunf (1-\Psi_{\delta,\eta})\hmun,
\end{array}
\]
and
\[
\begin{array}{l}
\dys\int_{Q}\a{\varphi}\cdot\nabla\tkunf (1-\Psi_{\delta,\eta})\hmun
\\\\ \dys=\int_{Q}\a{\varphi}\cdot \nabla T_k (u-\varphi)+\omega(n,m,\eta),
\end{array}
\]
while the first term can be handled by Fatou's lemma finally obtaining
\[
\dys\int_{Q}\a{u}\cdot\nabla  T_k (u-\varphi)
 \dys\leq \liminf_{\eta\to 0^+}\liminf_{m\to\infty}\liminf_{n\to\infty}\ \rifer{b1}.
\]
We now deal with \rifer{a1}. Let us define $\Theta_{k,m}(s)$ as the primitive function of $T_k (s) h_m(s)$, observe that $\Theta_{k,m}$ is a bounded function; so that thanks to Lemma \ref{acp}, for any $\eta>0$ there exists $\delta$ small enough such that
\[
\begin{array}{l}
\dys\left| \int_{Q}\Theta_{k,m} (u_n-\varphi)\hmun(\Psi_{\delta})_t\right|=\int_{Q}\Theta_{k} (u-\varphi)|(\Psi_{\delta})_t |+\omega (n)\\\\
\dys \leq \eta +\omega(n) =\omega(n,\eta),
\end{array}
\]
and so finally
\[
\begin{array}{l}
\rifer{a1}=\dys\int_{0}^{T}\langle (u_{n}-\varphi)_t,\tkunf(1-\Psi_{\delta,\eta})\hmun\rangle\ dt
\\\\ \dys+\int_{0}^{T}\langle \varphi_t,\tkunf(1-\Psi_{\delta,\eta})\hmun\rangle\ dt\\\\
\dys=\into \Theta_{k,m} (u_n -\varphi)(T) -\into \Theta_{k,m} ( -\varphi)(0)+\int_{Q}\Theta_{k,m} (u_n-\varphi)(\Psi_{\delta})_t \\\\+\dys\int_{0}^{T}\langle \varphi_t,\tkunf(1-\Psi_{\delta,\eta})\hmun\rangle\ dt \geq \into \Theta_k (u -\varphi)(T) \\\\
\dys-\into \Theta_k (-\varphi)(0)+\dys\int_{0}^{T}\langle \varphi_t, T_k (u-\varphi)\rangle\ dt +\omega(n,m,\eta),
\end{array}
\]
where in the last passage we used the fact that $r>p$ and Fatou's lemma which can be applied for almost every $0\leq T' \leq T$.
Passing to the limit and  gathering together all these facts we can conclude that $u$ is an entropy solution of \rife{lim1}. Actually we proved this fact for almost every $0\leq T'\leq T$ but thanks to uniqueness of the entropy solution    one can easily show that $u$ is the  entropy solution  for any $T>0$.

To prove \rife{rico} take $\psi\in C^{\infty}_{0}(Q)$ in \rife{ane1} to obtain
\[
\int_{Q}|u_n|^{q-1}u_n \psi= -\int_{Q}\a{u}\cdot\nabla\psi+\int_{Q}g\psi+\int_{Q}\psi\ d\lambda +\omega(n),
\]
which together with the fact that $u$ is an entropy solution of problem \rife{lim1} (and so a distributional one) yields \rife{rico} for $\psi$ smooth. Finally,  an easy density argument allows us to conclude the proof.
\end{proof}

\end{document}